\documentclass[12pt,twoside]{amsart}
\usepackage{amssymb}
\usepackage{verbatim}
\usepackage{amsmath}
\usepackage{bm}
\usepackage{a4wide}
\usepackage[latin1]{inputenc}
\usepackage[T1]{fontenc}
\usepackage{times}
\usepackage{amssymb,latexsym}
\usepackage{enumerate}
\usepackage[stable]{footmisc}
\usepackage{color}
\usepackage[colorlinks,linkcolor=black,citecolor=black, urlcolor=black]{hyperref}
\newtheoremstyle{mythm}{2ex plus 1ex minus .2ex}{2ex plus 1ex minus .2ex}     {\itshape}{}{\bfseries}{.}{0.7em}{}
\theoremstyle{mythm}
\newtheorem{prop}{Proposition}
\newtheorem{thm}[prop]{Theorem}
\newtheorem{lem}[prop]{Lemma}

\newtheoremstyle{mythm1}{2ex plus 1ex minus .2ex}{2ex plus 1ex minus .2ex}     {\normalfont}{}{\bfseries}{.}{1em}{} 
\theoremstyle{mythm1}

\begin{document}

\thanks{Supported by National Natural Science Foundation of China, No. 12271101}

\address [Yuanpu Xiong] {Department of Mathematical Sciences, Fudan University, Shanghai, 200433, China}
\email{ypxiong@fudan.edu.cn}

\title{Minimal $L^2$ and $L^p$ Ohsawa-Takegoshi extensions}
\author{Yuanpu Xiong}

\date{}

\begin{abstract}
We find a precise relationship between the minimal extensions in $L^2$ and $L^p$ Ohsawa-Takegoshi extension theorems. This relationship also gives another proof to the $L^p$ version of the Ohsawa-Takegoshi extension theorem, which is different from the original proof due to Berndtsson-P\u{a}un.
\end{abstract}

\maketitle

Let $\Omega$ be a bounded pseudoconvex domain in $\mathbb{C}^n$, $S$ a closed complex submanifold of codimension $k$ and $\varphi$ a psh function on $\Omega$. For any $p>0$, denote the weighted $p$-Bergman space by
\[
A^p(\Omega,\varphi):=\left\{f\in\mathcal{O}(\Omega);\ \int_\Omega|f|^pe^{-\varphi}<\infty\right\}.
\]
The following sharp Ohsawa-Takegoshi extension theorem is presented in \cite{BL} (see also \cite{OT,Blocki,GZ}).

\begin{thm}\label{th:OT_2}
If $f$ is a holomorphic function on $S$ with $
\int_S|f|^2e^{-\varphi+kB}<\infty$, then there exists $F\in{A^2(\Omega,\varphi)}$, such that $F|_S=f$ and
\begin{equation}\label{eq:OT}
\int_\Omega|F|^2e^{-\varphi}\leq\sigma_k\int_S|f|^2e^{-\varphi+kB},
\end{equation}
where $\sigma_k$ is the volume of the unit ball in $\mathbb{C}^k$ and $B$ is a continuous function on $\Omega$ determined by the geometry of $\Omega$ and $S$.
\end{thm}

We shall not give the precise definition of $B$, because it is not used here. The interested reader may consult \cite{BL}.

An $L^p$ version ($0<p<2$) of Ohsawa-Takegoshi extension theorem also holds. It is originally proved by Berndtsson-P\u{a}un \cite{BP08,BP10} (see also \cite{GZ}). As Theorem \ref{th:OT_2}, we write the theorem as follows.

\begin{thm}\label{th:OT_p}
Let $0<p<2$. If $f$ is a holomorphic function on $S$ with $\int_\Omega|f|^pe^{-\varphi+kB}<\infty$, then there exists $F\in{A^p(\Omega,\varphi)}$, such that $F|_S=f$ and
\begin{equation}\label{eq:OT}
\int_\Omega|F|^pe^{-\varphi}\leq\sigma_k\int_S|f|^pe^{-\varphi+kB},
\end{equation}
where $\sigma_k$ is the volume of the unit ball in $\mathbb{C}^k$ and $B$ is a continuous function on $\Omega$ determined by the geometry of $\Omega$ and $S$.
\end{thm}

Theorem \ref{th:OT_p} is proved by an ingenious iterative application of Theorem \ref{th:OT_2}. The weight function of the form $\varphi+(2-p)\log|F|$ is used, where $F$ is an extension in Theorem \ref{th:OT_2}. In this note, we shall present another proof of Theorem \ref{th:OT_p} without using iteration. It also relies on Theorem \ref{th:OT_2}, and makes use of extensions with minimal weighted $L^2$ and $L^p$ norms in Theorem \ref{th:OT_2} and Theorem \ref{th:OT_p}. Moreover, we use the calculus of variations (which is inspired by \cite{CZ}) to deduce an essential property of the minimal extensions. The argument actually works for more general settings where the above theorems hold. But we choose to follow the philosophy in \cite{BL}: restrict to the case that $\Omega$ is a domain in $\mathbb{C}^n$ in order to emphasize the basic ideas.

By exhaustion and approximation as usual, we may assume that $\varphi$ is a smooth psh function in a neighbourhood of $\overline{\Omega}$ and $f$ is the restriction to $S$ of a holomorphic function $f_0$ in a neighbourhood of $\overline{\Omega}$. Consider the set
\[
E^p(f,S,\varphi):=\{g\in{A^p(\Omega,\varphi)};\ g|_S=f\}.
\]
Clearly, $E^p(f,S,\varphi)\neq\varnothing$ (since $f_0\in{E^p(f,S)}$). It follows from a normal family argument that the minimizing problem
\[
m_p(f,S,\varphi):=\min\left\{\int_\Omega|g|^pe^{-\varphi};\ g\in{E^p(f,S)}\right\}
\]
admits minimizers. It seems reasonable to say a minimizer is a \textit{minimal extension}. Note that the minimal extension is unique when $p\geq1$ for $A^p(\Omega,\varphi)$ is strongly convex (cf. \cite{CZ}, \S 2.1). But we do not know whether it is unique when $0<p<1$. Nontheless, we may fix one minimal extension in either cases, and denote it by $F_{p,\varphi}$. Therefore, if $h\in{A^p(\Omega,\varphi)}$ and $h|_S=0$, then
\begin{equation}\label{eq:minimal_1}
\int_\Omega|F_{p,\varphi}+th|^pe^{-\varphi}\geq\int_\Omega|F_{p,\varphi}|^pe^{-\varphi},\ \ \ \forall\,t\in\mathbb{C}.
\end{equation}

Apply the calculus of variations as \cite{CZ}, we see that
\begin{equation}\label{eq:variation_p}
\int_\Omega|F_{p,\varphi}|^{p-2}\overline{F_{p,\varphi}}he^{-\varphi}=0
\end{equation}
for all $h\in{A^p(\Omega,\varphi)}$ with $h|_S=0$ when $p\geq1$. To handle the case $0<p<1$, set $\widetilde{\varphi}_p:=\varphi+(2-p)\log|F_{p,\varphi}|$. We first show that

\begin{lem}\label{lm:including}
If $0<p<2$, then $A^2(\Omega,\widetilde{\varphi}_p)\subset{A^p(\Omega,\varphi)}$.
\end{lem}

\begin{proof}
This follows immediately from H\"{o}lder's inequality:
\begin{align*}
\int_\Omega|f|^pe^{-\varphi}
\leq&\,\left(\int_\Omega|F_{p,\varphi}|^{p-2}|f|^2e^{-\varphi}\right)^{p/2}\left(\int_\Omega|F_{p,\varphi}|^pe^{-\varphi}\right)^{1-p/2}\\
=&\,\left(\int_\Omega|f|^2e^{-\widetilde{\varphi}_p}\right)^{p/2}\left(\int_\Omega|F_{p,\varphi}|^pe^{-\varphi}\right)^{1-p/2}.\qedhere
\end{align*}
\end{proof}

The following lemma is also elementary.

\begin{lem}\label{lm:ineq}
If\/ $0<\alpha<1$, then $(1+x)^\alpha\leq1+\alpha{x}$ for all $x\in[-1,\infty)$.
\end{lem}

\begin{proof}
Let $\mu(x)=(1+x)^\alpha-1-\alpha{x}$. We have $\mu(0)=0$, $\mu'(0)=0$ and $\mu''(x)<0$, which implies $\mu\leq0$.
\end{proof}

Next, we always assume that $h\in{A^2(\Omega,\widetilde{\varphi}_p)}$ and $h|_S=0$. It follows from Lemma \ref{lm:ineq} (with $\alpha=p/2$) that
\begin{align}
|F_{p,\varphi}+th|^p
=&\,\left(|F_{p,\varphi}|^2+2\mathrm{Re}\,(t\overline{F_{p,\varphi}}h)+|t|^2|h|^2\right)^{p/2}\nonumber\\
=&\,|F_{p,\varphi}|^p\left(1+2\mathrm{Re}\,(t\overline{F_{p,\varphi}}h/|F_{p,\varphi}|^2)+|t|^2|h|^2/|F_{p,\varphi}|^2\right)^{p/2}\nonumber\\
\leq&\,|F_{p,\varphi}|^p+p|F_{p,\varphi}|^{p-2}\mathrm{Re}\,(t\overline{F_{p,\varphi}}h)+\frac{p}{2}|t|^2|F_{p,\varphi}|^{p-2}|h|^2\label{eq:minimal_2}.
\end{align}
on $\Omega\setminus\{F_{p,\varphi}=0\}$. Since the Lebesgue measure of $\{F_{p,\varphi}=0\}$ is zero, we infer from \eqref{eq:minimal_1} and \eqref{eq:minimal_2} that
\[
2\int_\Omega|F_{p,\varphi}|^{p-2}\mathrm{Re}\,(t\overline{F_{p,\varphi}}h)e^{-\varphi}+|t|^2\int_\Omega|F_{p,\varphi}|^{p-2}|h|^2e^{-\varphi}\geq0,\ \ \ \forall\,t\in\mathbb{C}.
\]
Letting $t\rightarrow0$ along the real and imagery axis, we obtain \eqref{eq:variation_p} for all $h\in{A^2(\Omega,\widetilde{\varphi}_p)}$ with $h|_S=0$ when $0<p<1$ (and hence when $0<p<2$ in view of Lemma \ref{lm:including}). 

Similarly, set
\[
E^2(f,S,\widetilde{\varphi}_p):=\{g\in{A^2(\Omega,\widetilde{\varphi}_p)};\ g|_S=f\}.
\]
Note that $\widetilde{\varphi}_p$ is \textit{not} smooth when $F_{p,\varphi}$ has zeros. Thus it is not clear whether $f_0\in{E^2(f,S,\widetilde{\varphi}_p)}$. But when $0<p<2$, $\widetilde{\varphi}_p$ is still psh. It is a consequence of the Theorem \ref{th:OT_2} that $E^2(f,S,\widetilde{\varphi}_p)\neq\varnothing$. More precisely, we have
\begin{equation}\label{eq:f_p_2}
\int_S|f|^2e^{-\widetilde{\varphi}_p+kB}=\int_S|F_{p,\varphi}|^{p-2}|f|^2e^{-\varphi+kB}=\int_S|f|^pe^{-\varphi+kB}<\infty
\end{equation}
for $F_{p,\varphi}=f$ on $S$. Thus Theorem \ref{th:OT_2} yields an extension $g_0\in{A^2(\Omega,\widetilde{\varphi}_p)}$ with
\begin{equation}\label{eq:OT_phi_tilde}
\int_\Omega|g_0|^2e^{-\widetilde{\varphi}_p}\leq\sigma_k\int_S|f|^2e^{-\widetilde{\varphi}_p+kB}<\infty,
\end{equation}
so that $g_0\in{E^2(f,S,\widetilde{\varphi}_p)}$.

One can proceed analogously to obtain a minimal extension $F_{2,\widetilde{\varphi}_p}\in{A^2(\Omega,\widetilde{\varphi}_p)}$, such that
\[
\int_\Omega|F_{2,\widetilde{\varphi}_p}+th|^2e^{-\varphi}\geq\int_\Omega|F_{2,\widetilde{\varphi}_p}|^2e^{-\varphi},\ \ \ \forall\,t\in\mathbb{C}.
\]
when $h\in{A^2(\Omega,\widetilde{\varphi}_p)}$ and $h|_S=0$. For such $h$, we have
\begin{equation}\label{eq:variation_2}
\int_\Omega\overline{F_{2,\widetilde{\varphi}_p}}he^{-\widetilde{\varphi}_p}=\int_\Omega|F_{p,\varphi}|^{p-2}\overline{F_{2,\widetilde{\varphi}_p}}he^{-\varphi}=0.
\end{equation}
Thus \eqref{eq:variation_p} and \eqref{eq:variation_2} imply that
\begin{equation}\label{eq:variation_difference}
\int_\Omega|F_{p,\varphi}|^{p-2}\overline{F_{p,\varphi}-F_{2,\widetilde{\varphi}_p}}h=0
\end{equation}
for all $h\in{A^2(\Omega,\widetilde{\varphi}_p)}$ with $h|_S=0$. Since
\[
\int_\Omega|F_{p,\varphi}|^2e^{-\widetilde{\varphi}_p}=\int_\Omega|F_{p,\varphi}|^pe^{-\varphi}<\infty,
\]
i.e., $F_{p,\varphi}\in{A^2(\Omega,\widetilde{\varphi}_p)}$, we may take $h=F_{p,\varphi}-F_{2,\widetilde{\varphi}_p}$ in \eqref{eq:variation_difference}, which gives
\begin{equation}\label{eq:relation_minimal_extension}
F_{p,\varphi}=F_{2,\widetilde{\varphi}_p}.
\end{equation}
This together with \eqref{eq:f_p_2} and \eqref{eq:OT_phi_tilde} imply that
\begin{align*}
\int_\Omega|F_{p,\varphi}|^pe^{-\varphi}
=&\,\int_\Omega|F_{p,\varphi}|^2e^{-\widetilde{\varphi}_p}=\int_\Omega|F_{2,\widetilde{\varphi}_p}|^2e^{-\widetilde{\varphi}_p}\leq\sigma_k\int_S|f|^2e^{-\widetilde{\varphi}_p+kB}=\sigma_k\int_S|f|^pe^{-\varphi+kB},
\end{align*}
which is the desired estimate and proves Theorem \ref{th:OT_p}.

Note that the relationship \eqref{eq:relation_minimal_extension} holds when $f,\varphi,\Omega$ and $S$ satisfy some additional conditions. Indeed, in the above proof, these conditions are only used to guarantee that $E^p(f,S,\varphi)\neq\varnothing$ (i.e., $f_0\in{E^p(f,S,\varphi)}$). If we are allowed to use Theorem \ref{th:OT_p}, then $E^p(f,S,\varphi)\neq\varnothing$ as long as $\int_S|f|^pe^{-\varphi+kB}<\infty$. Thus we have proved that

\begin{thm}\label{th:minimal_extension}
Let $0<p<2$. Let $\Omega$ be a bounded pseudoconvex domain in $\mathbb{C}^n$, $S$ a complex submanifold of codimension $k$ and $\varphi$ a psh function on $\Omega$. If $f$ is a holomorphic function on $S$ with $\int_S|f|^pe^{-\varphi+kB}<\infty$, then any minimimal extension $F_{p,\varphi}\in{A^p(\Omega,\varphi)}$ of $f$ satisfies
\[
F_{p,\varphi}=F_{2,\widetilde{\varphi}_p},
\]
where $\widetilde{\varphi}_p=\varphi+(2-p)\log|F_{p,\varphi}|$
\end{thm}

\end{document}